\documentclass{article}
\usepackage[utf8]{inputenc}
\usepackage{amssymb, amsmath}
\usepackage{amsthm}
\usepackage{amsfonts} 
\usepackage{tikz-cd}

\usetikzlibrary{arrows}

\title{Counting Representations of Quivers Respecting Nilpotent Relations over Finite Fields}
\author{Bangming Deng and Jiuzhao Hua} 
\newtheorem{thm}{Theorem}[section]
\newtheorem{lem}{Lemma}[section]
\newtheorem{cor}{Corollary}[section]
\newtheorem{dfn}{Definition}[section]
\newtheorem{cjc}{Conjecture}[section]

\usepackage[parfill]{parskip}

\begin{document}
\date{\vspace{-0.5cm}}\date{} 
% Toggle commenting to test
\maketitle\begin{abstract}
This paper presents analogous results of Hua \cite{JH 2000}\cite{JH 2021} on numbers of representations of quivers over finite fields which respect nilpotent relations under certain assumptions. A closed formula which counts isomorphism classes of absolutely indecomposable representations with given dimension vectors is given and a $q$-deformation of Weyl-Kac denominator identity is established. In principle, if the numbers of representations are known, then the numbers of isomorphism classes of absolutely indecomposable representations are known.
\end{abstract}

\section{Introduction}

Let $\mathbb{N}$, $\mathbb{Z}$ and $\mathbb{Q}$ be the sets of non-negative integers, integers and rational numbers respectively. Given a positive integer $n$ and
an $n\times n$ matrix $C=[a_{ij}]$ with $a_{ij}\in\mathbb{N}$, let $\Gamma$ be the \textit{quiver} defined by $C$, i.e., $\Gamma$ is the directed graph with $n$ vertices $\{1, 2, \cdots, n\}$ equipped with $a_{ij}$ arrows from $i$ to $j$ ($1\le i,j \le n$). We attach an indeterminate $X_i$ to vertex $i$ ($1\le i\le n$) and $X_{ij}^{(k)}$ to the $k$-th arrow from $i$ to $j$ ($1\le i,j \le n$ and $1\le k \le a_{ij}$). Let $\Omega$ be the set of all arrows in $\Gamma$, thus $\Omega$ can be identified with the set 
$\{ X_{ij}^{(k)} | \,1\le i,j\le n$, $1\le k \le a_{ij} \text{ and } a_{ij} > 0\}$. The following is an example of a quiver and its defining matrix:
\[
\begin{tikzcd}[row sep=large, column sep = large]
\arrow[loop left, distance=3em, "{X_{11}^{(1)}}"]
\underset{1}{\circ}\arrow[r, bend left, "{X_{12}^{(1)}}"] & 
\arrow[l, bend left, "{X_{21}^{(1)}}" ] \underset{2}{\circ} \arrow[r, bend left, "{X_{23}^{(1)}}"] \arrow[r, bend right, swap, "{X_{23}^{(2)}}"] 
& \underset{3}{\circ}
\end{tikzcd},
\quad
\left[
\begin{array}{ccc}
1 & 1 & 0 \\ 
1 & 0 & 2 \\ 
0 & 0 & 0 
\end{array}
\right].
\]

Let $\mathbb{F}$ be a field. A monomial in the following form is called a \textit{simple relation} for $\Gamma$:
$$X_{i_0i_1}^{(k_1)}X_{i_1i_2}^{(k_2)}X_{i_2i_3}^{(k_3)}\cdots X_{i_{s-1} i_{s}}^{(k_s)},$$
where $i_0$ is called the \textit{starting point} of the relation and $i_{s}$ is called the \textit{ending point}. It is necessary that $a_{i_0i_1}a_{i_1i_2}a_{i_2i_3}\cdots a_{i_{s-1} i_{s}} \ne 0$ and $1\le k_m \le a_{i_{m-1}i_m}$ ($1\le m \le s$). It is evident that every simple relation is induced by a path in $\Gamma$. If $i_0 = i_s$ then the relation is called \textit{cyclic}. So simple cyclic relations are induced by oriented cycles in $\Gamma$. A \textit{relation} for $\Gamma$ is a linear combination over $\mathbb{F}$ of simple relations that share the same starting point and the same ending point. A relation is called \textit{cyclic} if all of its summands are cyclic. For example, the following is a cyclic relation for the quiver mentioned above:
$$X_{11}^{(1)}X_{12}^{(1)}X_{21}^{(1)} - X_{12}^{(1)}X_{21}^{(1)}X_{11}^{(1)}.$$
In this paper, we are only interested in cyclic relations. 

Given $\alpha = (\alpha _1, \cdots, \alpha_n)\in\mathbb{N}^n$, a \textit{representation of quiver} $\Gamma$ of dimension $\alpha$ over a field $\mathbb{F}$ is a function 
$\sigma: \Omega \mapsto \{ \textit{Matrices over } \mathbb{F} \}$ such that $\sigma(X_{ij}^{(k)})$ has order $\alpha_i \times\alpha_j$ for $1\le i , j\le n$ and 
$1\le k \le a_{ij}$. Note that matrices with $0$ rows or $0$ columns are permitted. $\alpha$ is called the \textit{dimension vector} of $\sigma$ and denoted by $\dim \sigma$. The following diagram defines a representation of the quiver mentioned above with 
dimension vector $(1,2,2$):
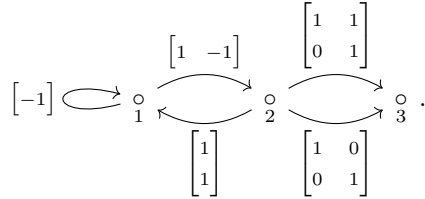
\[
\begin{tikzcd}[row sep=large, column sep = large, ampersand replacement=\&]
\arrow[loop left, distance=3em, "{\begin{bmatrix}-1 \end{bmatrix}}"]
\underset{1}{\circ}\arrow[r, bend left, "{\begin{bmatrix}1 & \!\! -1\end{bmatrix}}"] \& 
\arrow[l, bend left, "{\begin{bmatrix}1 \\ 1\end{bmatrix}}" ] \underset{2}{\circ} \arrow[r, bend left, "{\begin{bmatrix}1 & 1 \\0 & 1\end{bmatrix}}"] 
\arrow[r, bend right, swap, "{\begin{bmatrix}1 & 0 \\0 & 1\end{bmatrix}}"] 
\& \underset{3}{\circ}
\end{tikzcd}.
\]
Given two representations $\sigma$ and $\tau$ with dimension vectors 
$\alpha$ and $\beta$ respectively, an $n$-tuple $(H_1, \cdots, H_n)$ of
matrices over $\mathbb{F}$ is called a \textit{homomorphism} from $\sigma$ to $\tau$ if $H_i$ has order $\alpha_i \times \beta_i$ for $1\le i \le n$ and 
$\sigma(X_{ij}^{(k)}) H_j=H_i\tau(X_{ij}^{(k)})$ for all $1\le i , j\le n$ and $1\le k \le a_{ij}$. If $\dim \sigma = \dim \tau$ and $H_1, \cdots, H_n$ are all nonsingular, then $(H_1, \cdots, H_n)$ is called an \textit{isomorphism}. The \textit{direct sum} of $\sigma$ and $\tau$, denoted by $\sigma\oplus\tau$, is defined by 
$(\sigma\oplus\tau)(X_{ij}^{(k)}) = \sigma(X_{ij}^{(k)})\oplus\tau(X_{ij}^{(k)})$ ($1\le i , j\le n$ and $1\le k \le a_{ij}$). 

A representation is called 
\textit{decomposable} if it is isomorphic to a direct sum of two representations with non-zero dimension vectors. An indecomposable representation of $\Gamma$ over $\mathbb{F}$ 
is called \textit{absolutely indecomposable} if it is still indecomposable when it is considered as a representation of $\Gamma$ over $\overline{\mathbb{F}}$, the algebraic closure of $\mathbb{F}$.

Let $\text{Mat}(m\times n, \mathbb{F})$ be the set of all $m\times n$ matrices over $\mathbb{F}$ for $m, n\in\mathbb{N}$. Given $\alpha\in\mathbb{N}^n$,
let $\text{Rep}(\alpha, \mathbb{F})$ be the set of all representations of $\Gamma$ of dimension $\alpha$ over $\mathbb{F}$. Thus $\text{Rep}(\alpha, \mathbb{F})$ is naturally identified with the following affine variety:
$$\bigoplus_{X_{ij}^{(k)} \in \Omega}\text{Mat}(\alpha_i\!\times\!\alpha_j, \mathbb{F}).$$
Let $\text{GL}(m,\mathbb{F})$ be the General Linear Group of order $m$ over $\mathbb{F}$. Given $\alpha\in\mathbb{N}^n$,
% let $\text{GL}(\alpha,\mathbb{F}) = \text{GL}(\alpha_1,\mathbb{F}) \oplus \cdots \oplus \text{GL}(\alpha_n,\mathbb{F})$,
let $\text{GL}(\alpha,\mathbb{F}) = \prod_{i=1}^n\!\text{GL}(\alpha_i,\mathbb{F})$,
then the linear algebraic group $\text{GL}(\alpha,\mathbb{F})$ acts on $\text{Rep}(\alpha, \mathbb{F})$ as follows:
\begin{align*}
\text{GL}(\alpha,\mathbb{F}) \times \text{Rep}(\alpha, \mathbb{F}) & \to \text{Rep}(\alpha, \mathbb{F}) \\
(g, \sigma) &\mapsto g\sigma,
\end{align*}
where $g\sigma(X_{ij}^{(k)}) = g_i^{-1}\sigma(X_{ij}^{(k)})g_j$ for $g = (g_1,\cdots, g_n) \in \text{GL}(\alpha,\mathbb{F})$. It is obvious that two representations from
$\text{Rep}(\alpha,\mathbb{F})$ are isomorphic if and only if they are in the same orbit. 

Given a representation $\sigma$ of $\Gamma$ and a simple relation $R$ for $\Gamma$, where 
$$R=X_{i_0i_1}^{(k_1)}X_{i_1i_2}^{(k_2)}X_{i_2i_3}^{(k_3)}\cdots X_{i_{s-1} i_{s}}^{(k_s)},$$
$\sigma$ acts on $R$ naturally, i.e.,
$$\sigma(R) =\sigma\!\left(X_{i_0i_1}^{(k_1)}\right)\sigma\!\left(X_{i_1i_2}^{(k_2)}\right)\sigma\!\left(X_{i_2i_3}^{(k_3)}\right)\cdots \sigma\!\left(X_{i_{s-1}i_{s}}^{(k_s)}\right).$$
This action is naturally extended to any relation for $\Gamma$.
$\sigma$ is said to be \textit{respecting} $R$ if $\sigma(R)=0$; $\sigma$ is said to be \textit{respecting} $R$ \textit{nilpotently} if $\sigma(R)$ is a nilpotent matrix, i.e., 
$\sigma(R)^m=0$ for some $m\in\mathbb{N}$.

From now on, let $\mathbb{F} = \mathbb{F}_q$, the finite field with $q$ elements where $q$ is a prime power, $\mathcal{R}$ a set of cyclic relations for $\Gamma$. 
For $\alpha\in\mathbb{N}^n$, let $\text{Rep}(\alpha, \mathbb{F}_q)_\mathcal{R}$ be the set 
all representations of $\Gamma$ of dimension $\alpha$ 
over $\mathbb{F}_q$ that respect all relations in $\mathcal{R}$ nilpotently, i.e.,
$$
\text{Rep}(\alpha, \mathbb{F}_q)_\mathcal{R} = \{\sigma\in \text{Rep}(\alpha, \mathbb{F}_q) : \sigma(R) \textit{ is nilpotent for all } R\in\mathcal{R}\}.
$$
It is evident that $\text{Rep}(\alpha, \mathbb{F}_q)_\mathcal{R}$ is closed under the action of $\text{GL}(\alpha,\mathbb{F}_q)$.
In what follows, we assume that $|\text{Rep}(\alpha, \mathbb{F}_q)_\mathcal{R}|$ is a polynomial in $q$ with rational coefficients for any $\alpha\in\mathbb{N}^n$, i.e.,
there exists $r(\alpha,q) \in\mathbb{Q}[q]$ such that $|\text{Rep}(\alpha, \mathbb{F}_{q^d})_\mathcal{R}| = r(\alpha,q^d)$ for $d\ge 1$. 
Let $M(\alpha,q)$ ($I(\alpha, q)$, $A(\alpha,q)$) be the number of isomorphism classes of representations (indecomposable representations, absolutely indecomposable representations respectively) of $\Gamma$ of dimension $\alpha$ over $\mathbb{F}_q$ which respect all relations in $\mathcal{R}$. 

Counting formulae for $M(\alpha,q)$, $I(\alpha, q)$ and $A(\alpha,q)$ exist in Hua \cite{JH 2000} for quivers without relations and in Hua \cite{JH 2021} for the quiver with $1$ vertex and $g$ loops with nilpotent relations. It is widely known
that $A(\alpha,q)$'s are of significant importance because of their deep connections with Geometric Invariant Theory, Quantum Group Theory and Representation Theory of Kac-Moody Algebras (Kac \cite{VK 1983}, Ringel \cite{CR 1990} and Hausel \cite{TH 2010}). 

In case $\Gamma$ has no edge-loops, a theorem of Kac \cite{VK 1983} shows that the dimension vectors of absolutely indecomposable
representations of $\Gamma$ over $\mathbb{F}_q$ are precisely the positive roots of the root system of the Kac-Moody algebra
associated with $\Gamma$. Kac \cite{VK 1983} also conjectured that the constant term of the polynomial counting isomorphism classes of absolutely indecomposables with a given dimension vector is the same 
as the root multiplicity of the given dimension vector. This conjecture was proved by Crawley-Boevey and Van den Bergh \cite{C-V 2004} for indivisible dimension vectors and 
by Hausel \cite{TH 2010} in general. 
This paper proves analogous results of Hua \cite{JH 2000}\cite{JH 2021} on quivers with nilpotent relations under the assumption above.

\section{Numbers of Stabilizers}
This paper uses the same methodology as Hua \cite{JH 2000}. A key step is to determine the number of representations that are stabilized by conjugacy classes of $\text{GL}(\alpha, \mathbb{F}_q)$, which are parametrized by $n$-tuples of partitions and monic irreducible polynomials over $\mathbb{F}_q$.

A \textit{partition} $\lambda = (\lambda_1,\lambda_2,\cdots,\lambda_s)$ is a finite sequence of positive integers such that $\lambda_{i}\ge\lambda_{i+1}$ for $1\le i \le s-1$. 
The unique partition of $0$ is $(0)$. $|\lambda| :=\sum_{i\ge 1}\lambda_i$ is called the \textit{weight} of $\lambda$.
Let $\mathcal{P}$ be the set of partitions of all non-negative integers. 
Let $f(x) = a_0 + a_1x + a_2x^2+ \dots + a_{n-1}x^{n-1} + x^n\in \mathbb{F}_q[x]$ be a polynomial over $\mathbb{F}_q$ and $c(f)$ be its \textit{companion matrix}, i.e.,
$$
\newcommand*{\temp}{\multicolumn{1}{|}{}}
c(f) = 
\left[
\begin{array}{ccccc}
0 & 1 & 0 & \dots & 0 \\ 
0 & 0 & 1 & \dots & 0 \\ 
\vdots & \vdots & \vdots & \ddots & \vdots \\
0 & 0 & 0 & \dots & 1 \\
-a_0 & -a_1 & -a_2 & \dots & -a_{n-1} 
\end{array}
\right].
$$
For any $m\in\mathbb{N}\backslash\{0\}$, let $J_m(f)$ be the \textit{Jordan block matrix} of order $m$ with $c(f)$ on the main diagonal, i.e.,
$$
\newcommand*{\temp}{\multicolumn{1}{|}{}}
J_m(f) = 
\left[
\begin{array}{ccccc}
c(f) & I & 0 & \dots & 0 \\ 
0 & c(f) & I & \dots & 0 \\ 
\vdots & \vdots & \vdots & \ddots & \vdots \\
0 & 0 & 0 & \dots & I \\
0 & 0 & 0 & \dots & c(f) 
\end{array}
\right]_{m\times m}, 
$$
where $I$ is the identity matrix of order $\deg(f)$.
For $\lambda=(\lambda_1, \lambda_2, \dots, \lambda_s)\in\mathcal{P}$, let $J_{\lambda}(f)$ be the 
\textit{direct sum} of $J_{\lambda_i}(f)$ $(i=1,\dots,s)$, i.e.,
$$J_{\lambda}(f) = J_{\lambda_1}(f) \oplus J_{\lambda_2}(f) \oplus \dots \oplus J_{\lambda_s}(f),$$
which stands for
$$
\newcommand*{\temp}{\multicolumn{1}{|}{}}
\left[
\begin{array}{cccc}
J_{\lambda_1}(f) & 0 & \dots & 0 \\ 
0 & J_{\lambda_2}(f) & \dots & 0 \\ 
\vdots & \vdots & \ddots & \vdots \\
0 & 0 & \dots & J_{\lambda_s}(f)
\end{array}
\right].
$$

\begin{dfn}
For any matrix of order $m\times n$, the \textit{arm length} of index $(i,j)$ is one plus the number of minimal moves from $(i,j)$ to $(1,n)$, where diagonal moves are not permitted. Thus the arm length distribution is as follows:
$$
\left[
\begin{array}{llllll}
n & n-1 & \dots & 3 & 2 & 1\\ 
n+1 & n & \dots & 4 & 3 & 2\\ 
n+2 & n+1 & \dots & 5 & 4 & 3\\ 
\vdots & \vdots & \vdots & \vdots & \vdots & \vdots\\ 
m+n & m+n-1 & \dots & m+2 & m+1 & m
\end{array}
\right]_{m\times n}.
$$
The \textit{arm rank} of a matrix $M = [a_{ij}]$ of order $m\times n$, denoted by $ar(M)$, is the largest arm length of indexes of non-zero elements of $M$, i.e.,
$$ar(M) = \max\left\{\textit{arm length of }(i,j) \,|\, a_{ij} \ne 0 \textit{ where } 1\le i\le m, 1\le j\le n\right\}.$$
\end{dfn}

\begin{dfn}
A matrix $M = [a_{ij}]$ of order $m\times n$ is of type-U if it satisfies the following conditions:
\begin{itemize}
\item $a_{ij} = a_{st}$ if $(i,j)$ and $(s,t)$ have the same arm length,
\item the arm rank of $M$ is at most $\min\{m,n\}$.
\end{itemize}
\end{dfn}
Thus a type-U matrix of order $m\times n$ has either the following form when $m\ge n$:
$$
\newcommand*{\temp}{\multicolumn{1}{|}{}}
\left[
\begin{array}{lllll}
a_1 & a_2 & \dots & a_{n-1} & a_n \\ 
0 & a_1 & \dots & a_{n-2} & a_{n-1} \\ 
\vdots & \vdots & \ddots & \vdots & \vdots \\
0 & 0 & \dots & a_1 & a_2 \\
0 & 0 & \dots & 0 & a_1 \\
\cline{1-5}
0 & 0 & \dots & 0 & 0 \\
\vdots & \vdots & \vdots & \vdots & \vdots \\
0 & 0 & \dots & 0 & 0
\end{array}
\right]_{m \times n},
$$
or the following form when $m\le n$:
$$
\newcommand*{\temp}{\multicolumn{1}{|}{}}
\left[
\begin{array}{lllllllll}
0 & \dots & 0 & \temp & a_1 & a_2 & \dots & a_{m-1} & a_m \\ 
0 & \dots & 0 & \temp & 0 & a_1 & \dots & a_{m-2} & a_{m-1} \\ 
\vdots & \vdots & \vdots & \temp & \vdots & \vdots & \ddots & \vdots & \vdots \\
0 & \dots & 0 & \temp & 0 & 0 & \dots & a_1 & a_2 \\
0 & \dots & 0 & \temp & 0 & 0 & \dots & 0 & a_1 
\end{array}
\right]_{m \times n}. 
$$

\begin{thm}[Turnbull \& Aitken  \cite{T-A 1948}]\label{T-A Thm}
Let $\lambda=(\lambda_1, \lambda_2, \dots, \lambda_s)$ and $\mu=(\mu_1, \mu_2, \dots, \mu_t)$ be two partitions and $f(x)=x-a_0$ with $a_0\in\mathbb{F}_q$, then any matrix $U$ over $\mathbb{F}_q$ that satisfies 
$J_\lambda(f)U= UJ_\mu(f)$ can be written as an $s\times t$ block matrix in the following form:
$$
\newcommand*{\temp}{\multicolumn{1}{|}{}}
\left[
\begin{array}{cccc}
U_{11} & U_{12} & \dots & U_{1t} \\ 
U_{21} & U_{22} & \dots & U_{2t} \\ 
\vdots & \vdots & \ddots & \vdots \\
U_{s1} & U_{s2} & \dots & U_{st} 
\end{array}
\right],
$$
where each submatrix $U_{ij}$ is a type-U matrix over $\mathbb{F}_q$ of order $\lambda_i\times \mu_j$ for all $(i,j)$ where $1\le i \le s$ and $1\le j \le t$.
\end{thm}

As an example, let $\lambda = (3, 2, 2), \mu=(3,3,2)$ and $f(x) = x-t\in\mathbb{F}_q[x]$, then
$$\newcommand*{\temp}{\multicolumn{1}{|}{}}
J_\lambda(f)=\left[
\begin{array}{ccccccccc}
t & 1 & 0 & \temp & 0 & 0 & \temp & 0 & 0 \\ 
0 & t & 1 & \temp & 0 & 0 & \temp & 0 & 0 \\
0 & 0 & t & \temp & 0 & 0 & \temp & 0 & 0 \\
\cline{1-9}
0 & 0 & 0 & \temp & t & 1 & \temp & 0 & 0 \\ 
0 & 0 & 0 & \temp & 0 & t & \temp & 0 & 0 \\
\cline{1-9}
0 & 0 & 0 & \temp & 0 & 0 & \temp & t & 1 \\ 
0 & 0 & 0 & \temp & 0 & 0 & \temp & 0 & t 
\end{array}
\right]
, 
J_\mu(f)=\left[
\begin{array}{cccccccccc}
t & 1 & 0 & \temp & 0 & 0 & 0 & \temp & 0 & 0 \\ 
0 & t & 1 & \temp & 0 & 0 & 0 & \temp & 0 & 0 \\
0 & 0 & t & \temp & 0 & 0 & 0 & \temp & 0 & 0 \\
\cline{1-10}
0 & 0 & 0 & \temp & t & 1 & 0 & \temp & 0 & 0 \\ 
0 & 0 & 0 & \temp & 0 & t & 1 & \temp & 0 & 0 \\
0 & 0 & 0 & \temp & 0 & 0 & t & \temp & 0 & 0 \\
\cline{1-10}
0 & 0 & 0 & \temp & 0 & 0 & 0 & \temp & t & 1 \\ 
0 & 0 & 0 & \temp & 0 & 0 & 0 & \temp & 0 & t 
\end{array}
\right].
$$
Every matrix $U$ which satisfies $J_\lambda(f)\,U = UJ_\mu(f)$ can be written as a block matrix in the following form:
$$
\newcommand*{\temp}{\multicolumn{1}{|}{}}
U=\left[
\begin{array}{cccccccccc}
a & b & c & \temp & h & i & j & \temp & m & n \\ 
0 & a & b & \temp & 0 & h & i & \temp & 0 & m \\
0 & 0 & a & \temp & 0 & 0 & h & \temp & 0 & 0 \\
\cline{1-10}
0 & p & q & \temp & 0 & d & e & \temp & k & l \\ 
0 & 0 & p & \temp & 0 & 0 & d & \temp & 0 & k \\
\cline{1-10}
0 & u & v & \temp & 0 & r & s & \temp & f & g \\ 
0 & 0 & u & \temp & 0 & 0 & r & \temp & 0 & f 
\end{array}
\right].
$$

For a partition $\lambda\in\mathcal{P}$, let $m_\lambda^{\underline{i}}$ be the \textit{multiplicity} of $i$, i.e., $m_\lambda^{\underline{i}}$ is the number of parts equal to $i$ in $\lambda$, and $\lambda$ can be written in its ``\textit{exponential form}" $(1^{m_1}2^{m_2}3^{m_3}\cdots)$, where $m_i = m_\lambda^{\underline{i}}$.
Let $\lambda'=(\lambda_1', \lambda_2', \lambda_3', \dots)$ be the \textit{conjugate partition} of $\lambda$, which means that $\lambda_i'$ is the number of parts in $\lambda$ that are greater than or equal to $i$ for all $i\ge 1$. 
Let $\lambda$ and $\mu$ be two partitions, $\lambda'=(\lambda_1', \lambda_2', \lambda_3', \dots)$ and $\mu'=(\mu_1', \mu_2', \mu_3', \dots)$ be their conjugate partitions, we define two types of ``\textit{inner product}" of $\lambda$ and $\mu$ as follows:
\begin{equation}\label{inner prodt dfn}
\langle\lambda, \mu\rangle = \sum_{i\ge 1}\lambda_i'\mu_i' \textit{ and } (|\lambda,\mu|) =\langle\lambda,\mu\rangle - \sum_{s\ge 1}m_\lambda^{\underline{s}}m_\mu^{\underline{s}}.
\end{equation}
$\langle\lambda, \mu\rangle$ can also be expressed in the following forms:
\begin{equation}\label{inner prodt dfn alt}
\langle\lambda, \mu\rangle = \sum_{i,j\ge 1}\min(i,j)m_\lambda^{\underline{i}}m_\mu^{\underline{j}} = \sum_{i,j\ge 1}\min(\lambda_i,\mu_j).
\end{equation}
For an $n$-tuple of partitions $\pi = (\pi_1, \cdots, \pi_n)\in\mathcal{P}^n$ and $s\in\mathbb{N}\backslash\{0\}$, we define the \textit{multiplicity vector} of order $s$ by $d_\pi^{\,\underline{s}} := (m_{\pi_1}^{\underline{s}}, \cdots, m_{\pi_n}^{\underline{s}})$. Thus, $(|\pi_1|, \cdots, |\pi_n|) = \sum_{s\ge 1}s d_\pi^{\,\underline{s}}$.

\begin{cor}\label{cor 1}
Let $\lambda, \mu\in\mathcal{P}$ be two partitions and $f(x)=x-a_0$ with $a_0\in\mathbb{F}_q$, then the number of matrices $U$ over $\mathbb{F}_q$ that satisfy  $J_\lambda(f)U = UJ_\mu(f)$ is equal to $q^{\langle\lambda,\mu\rangle}$.
\end{cor}
\begin{proof}
This is a direct consequence of Theorem \ref{T-A Thm} and identity (\ref{inner prodt dfn alt}). An alternative proof is given in Hua \cite{JH 2000}.
\end{proof}

\begin{cor}\label{full stabilizer}
Given $\pi=(\pi_1,\cdots,\pi_n)\in\mathcal{P}^n$ and $f(x)=x-a_0$, let $\alpha=(|\pi_1|,\cdots,|\pi_n|)\in\mathbb{N}^n$ and $g=(J_{\pi_1}\!(f),\cdots,J_{\pi_n}\!(f))\in\textup{GL}(\alpha,\mathbb{F}_q)$. The stabilizer of $g$ in $\textup{Rep}(\alpha, \mathbb{F}_q)$ is defined as:
$$X_g=\left\{\sigma\in\textup{Rep}(\alpha,\mathbb{F}_q) : g\sigma = \sigma\right\}.$$
There holds:
$$ |X_g| = q^{\sum_{1\le i,j\le n} a_{ij}\langle\pi_i,\pi_j\rangle}.$$
\end{cor}
\begin{proof}
For $\sigma\in\text{Rep}(\alpha,\mathbb{F}_q)$, $\sigma\in X_g$ if and only if $ J_{\pi_i}\!(f)\sigma(X_{ij}^{(k)}) = \sigma(X_{ij}^{(k)}) J_{\pi_j}\!(f)$ 
for all $X_{ij}^{(k)} \in \Omega$,.
Thus
\begin{align*}
|X_g| &= \prod_{X_{ij}^{(k)}\in \Omega} |\left\{ U\in\text{Mat}(|\pi_i|\!\times\!|\pi_j|, \mathbb{F}_q) : J_{\pi_i}\!(f)U = UJ_{\pi_j}\!(f) \right\}|
\end{align*}
Corollary \ref{cor 1} implies that 
$$ |X_g| = \prod_{X_{ij}^{(k)}\in \Omega} q^{\langle\pi_i,\pi_j\rangle} 
	       = \prod_{1\le i,j\le n}q^{a_{ij}\langle\pi_i,\pi_j\rangle}
             = q^{\sum_{1\le i,j\le n} a_{ij}\langle\pi_i,\pi_j\rangle}. $$
\end{proof}

\begin{dfn}
Let $U=[u_{ij}]$ be a type-U matrix of order $m\times n$, the \textit{core} of $U$, denoted by $U_0$, is defined as follows:
\begin{align*}
    U_0=
    \begin{cases}
      \quad 0 &\!\!\!\text{matrix of order $m\times n$ if $m \ne n$}, \\
      u_{11}I, &\!\!\!\text{where $I$ is the identity matrix of order $n$ if $m = n$}.
    \end{cases}
 \end{align*}
Obviously, the core of a type-U matrix is also a type-U matrix.
If $U = [U_{ij}]$ is a block matrix of type-U matrices, then the \textit{core} of $U$, denoted by $U_0$, is the block matrix $[(U_{ij})_0]$. 
\end{dfn}
The following is an example of a 
block matrix of type-U matrices and its core:
$$
\newcommand*{\temp}{\multicolumn{1}{|}{}}
U=\left[
\begin{array}{cccccccccc}
a & b & c & \temp & h & i & j & \temp & m & n \\ 
0 & a & b & \temp & 0 & h & i & \temp & 0 & m \\
0 & 0 & a & \temp & 0 & 0 & h & \temp & 0 & 0 \\
\cline{1-10}
0 & p & q & \temp & 0 & d & e & \temp & k & l \\ 
0 & 0 & p & \temp & 0 & 0 & d & \temp & 0 & k \\
\cline{1-10}
0 & u & v & \temp & 0 & r & s & \temp & f & g \\ 
0 & 0 & u & \temp & 0 & 0 & r & \temp & 0 & f 
\end{array}
\right], 
U_0=\left[
\begin{array}{cccccccccc}
a & 0 & 0 & \temp & h & 0 & 0 & \temp & 0 & 0 \\ 
0 & a & 0 & \temp & 0 & h & 0 & \temp & 0 & 0 \\
0 & 0 & a & \temp & 0 & 0 & h & \temp & 0 & 0 \\
\cline{1-10}
0 & 0 & 0 & \temp & 0 & 0 & 0 & \temp & k & 0 \\ 
0 & 0 & 0 & \temp & 0 & 0 & 0 & \temp & 0 & k \\
\cline{1-10}
0 & 0 & 0 & \temp & 0 & 0 & 0 & \temp & f & 0 \\ 
0 & 0 & 0 & \temp & 0 & 0 & 0 & \temp & 0 & f 
\end{array}
\right].
$$

\begin{lem}\label{lemma 1}
Let $M$ be an $m\times n$ type-U matrix and $N$ an $n\times k$ type-U matrix, then MN is an $m \times k$ type-U matrix and $(MN)_0 = M_0N_0$.
\end{lem}

\begin{lem}\label{lemma 2}
Let $M=[M_{ij}]$ be an $m\times n$ block matrix of type-U matrices and $N=[N_{ij}]$ an $n\times k$ block matrix of type-U matrices such that $M$ and $N$ have compatible 
multiplication orders, i.e., the number of columns in $M_{is}$ is equal to
the number of rows in $N_{sj}$ for $1\le i \le m, 1 \le j \le k$ and $1\le s \le n$. Then MN is an $m \times k$ block matrix of type-U matrices and $(MN)_0 = M_0N_0$.
\end{lem}

The details of the proofs of the above two lemmas are left to the reader.

\begin{lem}\label{lemma 3}
Let $\lambda\in\mathcal{P}$, $f(x)=x-a_0\in\mathbb{F}_q[x]$ and $U$ a block matrix of type-U matrices satisfying $J_\lambda(f) U = UJ_\lambda(f)$.
Then $U$ is nilpotent if and only if $U_0$ is nilpotent.
\end{lem}
\begin{proof}
It can be shown by induction on the number of distinct parts in $\lambda$ that $U-U_0$ is always nilpotent. Suppose that $U$ is nilpotent, then $U^m=0$ for some $m\in\mathbb{N}$. Thus Lemma \ref{lemma 2} implies that
$(U_0)^m = (U^m)_0 = 0$. Thus $U_0$ is nilpotent. Conversely, suppose that $U_0$ is nilpotent, then $(U_0)^m=0$ for some $m\in\mathbb{N}$. Lemma \ref{lemma 2} implies that $(U^m)_0  = (U_0)^m = 0$.
Since $U^m - (U^m)_0$ is always nilpotent,  $U^m$ is nilpotent, and hence $U$ is nilpotent.
\end{proof}

\begin{thm}\label{degree 1}
Given $\pi=(\pi_1,\cdots,\pi_n)\in\mathcal{P}^n$ and $f(x)=x-a_0\in\mathbb{F}_q[x]$, let $\alpha=(|\pi_1|,\cdots,|\pi_n|)\in\mathbb{N}^n$ and $g=(J_{\pi_1}\!(f),\cdots,J_{\pi_n}\!(f))\in\textup{GL}(\alpha,\mathbb{F}_q)$, and
$X_g=\left\{\sigma\in\textup{Rep}(\alpha,\mathbb{F}_q) : g\sigma = \sigma\right\}$ the stabilizer of $g$ in $\textup{Rep}(\alpha, \mathbb{F}_q)$.
There holds:
$$ |X_g \cap \textup{Rep}(\alpha, \mathbb{F}_q)_\mathcal{R}| = q^{\sum_{1\le i,j\le n} a_{ij}(|\pi_i,\pi_j|)} \prod_{s\ge 1}\!r(d_{\pi}^{\,\underline{s}} , q),$$
where $d_\pi^{\,\underline{s}} = (m_{\pi_1}^{\underline{s}}, \cdots, m_{\pi_n}^{\underline{s}})$ is the multiplicity vector of order $s$ induced by $\pi$ and $r(d_{\pi}^{\,\underline{s}} , q) = |\textup{Rep}(d_{\pi}^{\,\underline{s}}, \mathbb{F}_{q})_\mathcal{R}|$ for $s\ge 1$.
\end{thm}
\begin{proof}

For any $\sigma\in X_g$, Theorem \ref{T-A Thm} implies that $\sigma(X_{ij}^{(k)})$ is a block matrix of type-U matrices for all
$X_{ij}^{(k)}\in \Omega$. The \textit{core} of $\sigma$ denoted by $\sigma_0$, is the representation of $\Gamma$ defined by:
$\sigma_0(X_{ij}^{(k)}) = \sigma(X_{ij}^{(k)})_0$ for all $1\le i , j\le n$ and $1\le k \le a_{ij}$. Obviously, $\sigma_0\in X_g$.

Let $R$ be a simple cyclic relation for $\Gamma$ and assume that the starting point (also the ending point) of $R$ is $e$. Then $\sigma(R)$ is a block matrix of type-U matrices 
and it satisfies $J_{\pi_e}\!(f) \sigma(R) = \sigma(R) J_{\pi_e}\!(f)$. Lemma \ref{lemma 2} implies that $\sigma(R)_0 = \sigma_0(R)$ and Lemma \ref{lemma 3} implies that 
$\sigma(R)$ is nilpotent if and only if $\sigma_0(R)$ is nilpotent. This equivalence still holds when $R$ is a relation for $\Gamma$. Thus we have
$$\sigma \in X_g \cap \textup{Rep}(\alpha, \mathbb{F}_q)_\mathcal{R} \textit{ if and only if } \sigma_0 \in X_g \cap \textup{Rep}(\alpha, \mathbb{F}_q)_\mathcal{R}.$$

Let $\pi^{(s)} = (\pi_1^{(s)}, \cdots, \pi_n^{(s)}) := (s^{m_{\pi_1}^{\underline{s}}}, \cdots, s^{m_{\pi_n}^{\underline{s}}}) \in \mathcal{P}^n$ for $s\ge 1$, where each component 
$s^{m_{\pi_i}^{\underline{s}}}$ ($i=1,\cdots,n$) is a partition in its ``exponential form'',
$\alpha^{(s)} = (sm_{\pi_1}^{\underline{s}}, \cdots, sm_{\pi_n}^{\underline{s}}) \in \mathbb{N}^n$ and
$g^{(s)} = (J_{\pi_1^{(s)}}\!(f), \cdots, J_{\pi_n^{(s)}}\!(f)) \in \text{GL}(\alpha^{(s)}, \mathbb{F}_q)$.

Since every matrix $\sigma_0(X_{ij}^{(k)})$ for $X_{ij}^{(k)}\in \Omega$ is a block matrix and all of its non-square submatrices are $0$,  $\sigma_0$ can be written as a direct sum of representations of $\Gamma$ in the following form:
$$\sigma_0 \cong \oplus_{s\ge 1}\tau^{(s)}, $$
where $\tau^{(s)} \in \text{Rep}(\alpha^{(s)}, \mathbb{F}_q)$ and $\tau^{(s)} \in X_{g^{(s)}}$ for all $s\ge 1$. There are only finitely many terms in the above sum because 
$(|\pi_1|, \cdots, |\pi_n|) = \dim \sigma_0 =  \sum_{s\ge 1}\dim \tau^{(s)}$.

Treating every matrix $\tau^{(s)}(X_{ij}^{(k)})$ for $X_{ij}^{(k)}\in \Omega$ as a linear transformation between vector spaces, and applying a base change in the underlying vector space for vertex $v$ 
($1\le v \le n$) which has dimension $sm_{\pi_v}^{\underline{s}}$ by the following mapping:
$$
\left[
\arraycolsep=1.4pt
\begin{array}{llll}
v_1 & v_2  & \cdots & v_m \\ 
v_{m+1} & v_{m+2} & \cdots & v_{m+m} \\ 
\vdots & \vdots & \cdots & \vdots \\
v_{(s-1)m+1} & v_{(s-1)m+2} & \cdots & v_{(s-1)m+m} 
\end{array}
\right]
\mapsto
\left[
\arraycolsep=3pt
\begin{array}{llll}
v_1 & v_{s+1} & \cdots & v_{(m-1)s+1} \\ 
v_2 & v_{s+2} & \cdots & v_{(m-1)s+2} \\ 
\vdots & \vdots & \cdots & \vdots \\
v_{s} & v_{s+s}  & \cdots & v_{(m-1)s+s} 
\end{array}
\right],
$$
where $m=m_{\pi_v}^{\underline{s}}$, it transforms $\tau^{(s)}$ into $s$ copies of identical representations:
$$\tau^{(s)} \, \cong \, \underbrace{\delta^{(s)} \oplus \cdots \oplus \delta^{(s)}}_\text{$s$ copies},$$
where $\dim \delta^{(s)} = (m_{\pi_1}^{\underline{s}}, \cdots, m_{\pi_n}^{\underline{s}}) = d_{\pi}^{\,\underline{s}}.$

For example, let $\Gamma = \tilde{A}_3$, the quiver with 4 vertices and 4 arrows which form a loop, $\pi=((2^23^2),(2^24^1),(2^13^1),(1^22^1) )\in\mathcal{P}^4$, 
then $\pi^{(2)} = (2^2, 2^2, 2^1, 2^1) \in \mathcal{P}^4$, $g^{(2)} = (J_{(2^2)}\!(f), J_{(2^2)}\!(f), J_{(2^1)}\!(f), J_{(2^1)}\!(f))$, every 
$\tau^{(2)}\in X_{g^{(2)}}$ should have the form as the left diagram below.
After bases are changed in the underlying vector spaces  as described, the representation on the left can be transformed into the representation on the right:
\[
\begin{tikzcd}[row sep=huge, column sep = huge, ampersand replacement=\&]
\overset{1}{\circ} \arrow[r, "{\newcommand*{\temp}{\multicolumn{1}{|}{}}
\left(\arraycolsep=2pt \def\arraystretch{0.8} \begin{array}{ccccc}
a & 0 & \temp & b & 0 \\ 
0 & a & \temp & 0 & b  \\ 
\cline{1-5} 
c & 0 & \temp & d & 0\\ 
0 & c & \temp & 0 & d
\end{array}\right)}"] 
\& \overset{2}{\circ} \arrow[d, "{\newcommand*{\temp}{\multicolumn{1}{|}{}}
\left(\arraycolsep=2pt \def\arraystretch{0.8} \begin{array}{cc}
e & 0 \\ 
0 & e \\ 
\cline{1-2} 
f & 0 \\ 
0 & f 
\end{array}\right)}" ] \\
\underset{4}{\circ} \arrow[u, "{\newcommand*{\temp}{\multicolumn{1}{|}{}}
\left(\arraycolsep=2pt \def\arraystretch{0.8} \begin{array}{ccccc}
h & 0 & \temp & i & 0 \\ 
0 & h & \temp & 0 & i  
\end{array}\right)}" ]
\& \underset{3}{\circ} \arrow[l, "{\newcommand*{\temp}{\multicolumn{1}{|}{}}
\left(\arraycolsep=2pt \def\arraystretch{0.8} \begin{array}{cc}
g & 0 \\ 
0 & g 
\end{array}\right)}"]
\end{tikzcd}
%\Longrightarrow
\mapsto
\begin{tikzcd}[row sep=huge, column sep = huge, ampersand replacement=\&]
\overset{1}{\circ} \arrow[r, "{\newcommand*{\temp}{\multicolumn{1}{|}{}}
\left(\arraycolsep=2pt \def\arraystretch{0.8} \begin{array}{ccccc}
a & b & \temp & 0 & 0 \\ 
c & d & \temp & 0 & 0 \\ 
\cline{1-5} 
0 & 0 & \temp & a & b \\ 
0 & 0 & \temp & c & d
\end{array}\right)}"] 
\& \overset{2}{\circ} \arrow[d, "{\newcommand*{\temp}{\multicolumn{1}{|}{}}
\left(\arraycolsep=2pt \def\arraystretch{0.8} \begin{array}{ccc}
e & \temp & 0 \\ 
f & \temp & 0 \\ 
\cline{1-3} 
0 & \temp & e \\ 
0 & \temp & f 
\end{array}\right)}" ] \\
\underset{4}{\circ} \arrow[u, "{\newcommand*{\temp}{\multicolumn{1}{|}{}}
\left(\arraycolsep=2pt \def\arraystretch{0.8} \begin{array}{ccccc}
h & i & \temp & 0 & 0 \\ 
\cline{1-5} 
0 & 0 & \temp & h & i  
\end{array}\right)}" ]
\& \underset{3}{\circ} \arrow[l, "{\newcommand*{\temp}{\multicolumn{1}{|}{}}
\left(\arraycolsep=2pt \def\arraystretch{0.8} \begin{array}{ccc}
g & \temp & 0 \\ 
\cline{1-3} 
0 & \temp & g 
\end{array}\right)}"]
\end{tikzcd}.
\]

It follows that $\sigma_0$ respects a relation $R$ for $\Gamma$ nilpotently if and only if $\tau^{(s)}$ respects $R$ nilpotently for all $s\ge 1$, if and only if
$\delta^{(s)}$ respects $R$ nilpotently for all $s\ge 1$, i.e., 
$$
\sigma_0 \in X_g \cap \textup{Rep}(\alpha, \mathbb{F}_q)_\mathcal{R} \textit{ if and only if }
\delta^{(s)} \in \textup{Rep}(d_{\pi}^{\,\underline{s}}, \mathbb{F}_q)_\mathcal{R} \textit{ for all } s\ge 1.
$$
Since $|\textup{Rep}(d_{\pi}^{\,\underline{s}}, \mathbb{F}_q)_\mathcal{R}| = r(d_{\pi}^{\,\underline{s}} , q)$, Corollary \ref{full stabilizer} implies that 
\begin{align*}
|X_g \cap \textup{Rep}(\alpha, \mathbb{F}_q)_\mathcal{R}| & = q^{\sum_{1\le i,j\le n}(a_{ij}\langle\pi_i,\pi_j\rangle - a_{ij}\sum_{s\ge 1}m_{\pi_i}^{\underline{s}}m_{\pi_j}^{\underline{s}})} \prod_{s\ge 1}\!r(d_{\pi}^{\,\underline{s}} , q)\\
& = q^{\sum_{1\le i,j\le n} a_{ij}(|\pi_i,\pi_j|)}  \prod_{s\ge 1}\!r(d_{\pi}^{\,\underline{s}} , q).
\end{align*}
\end{proof}

\begin{thm}\label{stabilizers}
Given $\pi=(\pi_1,\cdots,\pi_n)\in\mathcal{P}^n$ an $n$-tuple of partitions and $f(x)\in\mathbb{F}_q[x]$ a monic irreducible polynomial of degree $d$, 
let $\alpha=d(|\pi_1|,\cdots,|\pi_n|)$, $g=(J_{\pi_1}\!(f),\cdots,J_{\pi_n}\!(f))\in\textup{GL}(\alpha,\mathbb{F}_q)$ and
$X_g=\left\{\sigma\in\textup{Rep}(\alpha,\mathbb{F}_q) : g\sigma = \sigma\right\}$ the stabilizer of $g$ in $ \textup{Rep}(\alpha, \mathbb{F}_q)$.
There holds:
$$ |X_g \cap \textup{Rep}(\alpha, \mathbb{F}_q)_\mathcal{R}| = q^{d\sum_{1\le i,j\le n} a_{ij}(|\pi_i,\pi_j|)} \prod_{s\ge 1}\!r(d_{\pi}^{\,\underline{s}} , q^d) .$$
\end{thm}
\begin{proof}
Suppose that $d>1$ as the case for $d=1$ has been proved in Theorem \ref{degree 1}. Let $c(f)$ be the companion matrix for $f$ and $\langle c(f) \rangle$ be the subalgebra of $\textup{Mat}(d\times d, \mathbb{F}_q)$ generated by $c(f)$. Since $f$ is the characteristic equation of $c(f)$, $c(f)$ satisfies the polynomial $f$, i.e., $f(c(f))=0$. Since $f$ is irreducible, $f$ is the minimal polynomial satisfied by $c(f)$. This implies that 
$I, c(f), c(f)^2, \cdots, c(f)^{d-1}$ form a basis for $\langle c(f) \rangle$ over $\mathbb{F}_q$, i.e.,
$$\langle c(f) \rangle = \left\{\sum_{i=0}^{d-1}a_ic(f)^i \,|\, a_i\in\mathbb{F}_q, 0\le i \le d-1 \right\}.$$
Thus $\langle c(f) \rangle$ is a commutative subalgebra of $\textup{Mat}(d \times d, \mathbb{F}_q)$ and the following map is an isomorphism:
\begin{align*} 
\mathbb{F}_q[x]/(f(x)) &\to \langle c(f) \rangle \\
x &\mapsto c(f).
\end{align*}
Since $f$ is irreducible, $\mathbb{F}_q[x]/(f(x))$ is isomorphic to the finite field $\mathbb{F}_{q^d}$, and hence $\langle c(f) \rangle$ is a finite field with $q^d$ elements. 

When $\deg(f) > 1$, Theorem \ref{T-A Thm} still holds as long as all submatrices $U_{ij}$ take values from the finite field $\langle c(f) \rangle$. All arguments in the proof of Theorem \ref{degree 1} still work with $\mathbb{F}_q$ being replaced by $\langle c(f) \rangle$. Thus Theorem \ref{degree 1} implies the desired results.
\end{proof}

\section{Counting Formulae}
Let $\varphi_r(q)=(1-q)(1-q^2)\cdots(1-q^r)$ for $r\ge 1$ and $\varphi_0(q)=1$. For $\lambda = (1^{n_1}2^{n_2}3^{n_3}\cdots)\in\mathcal{P}$ in its ``\textit{exponential form}", we define $b_\lambda(q) = \prod_{i\ge1}\varphi_{n_i}(q)$.
Let $\phi_n(q)$ be the number of monic irreducible polynomials of degree $n$ in $\mathbb{F}_q[x]$ with $x$ excluded. It is known that for any positive integer $n$, 
\begin{equation}\phi_n(q) = \frac{1}{n}\sum_{d\,|\,n}\mu(d)(q^{\frac{n}{d}}-1),\end{equation}
where the sum runs over all divisors of $n$ and $\mu$ is the Möbius function.

\begin{dfn}
For $\pi = (\pi_1, \cdots, \pi_n)\in\mathcal{P}^n$, let $X^{|\pi|} = X_1^{|\pi_1|} \cdots X_n^{|\pi_n|}$ and $\mathbb{Q}(q)$ the field of rational functions in $q$ over
the rational field $\mathbb{Q}$. We define a formal power series in $\mathbb{Q}(q)[[X_1,\cdots,X_n]]$ as follows:
$$
P(X_1,\cdots,X_n, q) = \sum_{\pi\in\mathcal{P}^n} \!
\frac{q^{\sum_{1\le i, j \le n}\!a_{ij}(|\pi_i, \pi_j|)}\!\prod_{s\ge 1}\!r(d_{\pi}^{\,\underline{s}} , q)}
{\prod_{1\le i \le n}\!q^{\langle \pi_i, \pi_i\rangle} b_{\pi_i}\!(q^{-1})}X^{|\pi|}.
$$
Note that $((0), \cdots, (0))\in\mathcal{P}^n$ gives rise to a term equal to $1$ in the sum above.
\end{dfn}

\begin{thm} \label{burnside} For $\alpha=(\alpha_1, \cdots, \alpha_n)\in\mathbb{N}^n$, let $X^\alpha = X_1^{\alpha_1}\cdots X_n^{\alpha_n}$. There holds:
$$ 
\sum_{\alpha\in\mathbb{N}^n}^\infty M(\alpha, q)X^\alpha = \prod_{d=1}^\infty\left(
P(X_1^d,\cdots,X_n^d, q^d)
\right) ^{\phi_d(q)}.
$$
\end{thm}
\begin{proof}
The method applied in Theorem 4.3 from Hua \cite{JH 2000} still works here.
In current context, the Burnside orbit counting formula is applied to $\text{Rep}(\alpha, \mathbb{F}_q)_\mathcal{R}$ and the number of stabilizers for $X_g$ is given by
Theorem \ref{stabilizers}. Repeating the arguments there yields the desired result.
\end{proof}

\begin{dfn}
For $\alpha\in\mathbb{N}^n\backslash\{0\}$, let $\bar{\alpha} = \gcd(\alpha_1, \cdots, \alpha_n)$. 
Define rational functions $H(\alpha,q)$ for all $\alpha\in\mathbb{N}^n\backslash\{0\}$ as follows:
$$ 
\log\left(P(X_1,\cdots,X_n, q) \right) = \sum_{\alpha\in\mathbb{N}^n\backslash\{0\}}^\infty \!H(\alpha,q)X^\alpha ,
$$
where $\log$ is the formal logarithm, i.e., $\log(1+x) = \sum_{i\ge 1} (-1)^{i-1} x^i/i$.
\end{dfn}

\begin{thm}\label{A poly}The following identity holds for all $\alpha\in\mathbb{N}^n\backslash\{0\}$:
$$
A(\alpha,q) = (q-1)\sum_{d\,|\,\bar{\alpha}}\frac{\mu(d)}{d}H\Big(\frac{\alpha}{d}, q^d\Big),
$$
where the sum runs over all divisors of $\bar{\alpha}$.
\end{thm}
\begin{proof}
This is the counterpart of Theorem 4.6 from Hua \cite{JH 2000} with slight adjustment on the definition of $H(\alpha,q)$, same arguments apply.
\end{proof}

Analogues of Theorem 4.6 of Hua \cite{JH 2000} have been proved by Bozec, Schiffmann \& Vasserot \cite{B-S-V 2018} for Lusztig nilpotent varieties and their variants using techniques from Algebraic Geometry. Their definition of
nilpotency is stronger than the one used here. In the language of $\lambda$-ring and Adams operator, Theorem \ref{A poly} is equivalent to the following identities in the ring of formal power series 
$\mathbb{Q}(q)[[X_1,\cdots,X_n]]$:
\begin{align*}
\sum_{\alpha\in\mathbb{N}^n\backslash\{0\}}\!A(\alpha,q)X^\alpha =& \,\,(q-1)\text{Log}\left(P(X_1,\cdots,X_n,q)\right),  \\
P(X_1,\cdots,X_n,q) =& \,\,\text{Exp}\left(\frac{1}{q-1} \sum_{\alpha\in\mathbb{N}^n\backslash\{0\}}\!A(\alpha,q)X^\alpha\right).
\end{align*}
For the definitions of operator $\text{Log}$ and $\text{Exp}$, we refer to the Appendix in Mozgovoy \cite{SM 2007}.

Under the assumption that $r(\alpha, q)$ is a polynomial in $q$ with rational coefficients for all $\alpha\in\mathbb{N}^n$, 
$H(\alpha,q)$'s must be rational functions in $q$, so are $A(\alpha,q)$'s. As $A(\alpha,q)$'s take integer values for all prime powers $q$, $A(\alpha,q)$'s must be polynomials in $q$ with rational coefficients. It follows from Lemma 2.9 of Bozec, Schiffmann \& Vasserot \cite{B-S-V 2018} that $A(\alpha,q)\in\mathbb{Z}[q]$. Kac \cite{VK 1983} implies that the degree of $A(\alpha,q)$ 
is at most $1-\langle\alpha,\alpha\rangle$ where $\langle-,-\rangle$ is the Euler form defined by quiver $\Gamma$.

Theorem \ref{A poly} implies that if $r(\alpha,q)$'s are known for all $\alpha\in\mathbb{N}^n$ then $A(\alpha,q)$'s are known.
$I(\alpha,q)$ and $M(\alpha,q)$ can be calculated by the following identities:
\begin{equation}\label{count I}
I(\alpha,q) = \sum_{d\,|\,\bar{\alpha}}\frac{1}{d}\sum_{r\,|\,d}\mu\Big(\frac{d}{r}\Big)A\Big(\frac{\alpha}{d},q^r\Big),
\end{equation}
\begin{equation}\label{count M}
\sum_{\alpha\in\mathbb{N}^n} M(\alpha,q)X^\alpha = \prod_{\alpha\in\mathbb{N}^n\backslash\{0\}}^\infty(1-X^\alpha)^{-I(\alpha,q)}.
\end{equation}
Identity (\ref{count I}) is the counterpart of the first identity of Theorem 4.1 from Hua \cite{JH 2000} and identity (\ref{count M}) is a consequence of the Krull--Schmidt Theorem from representation theory. It follows that $I(\alpha,q)$ and $M(\alpha,q)$ are polynomials in $q$ with rational coefficients for all $\alpha\in\mathbb{N}^n$.  

\begin{thm}\label{gwki} Let $\Delta^+ = \{\alpha : A(\alpha,q) \ne 0, \alpha\in\mathbb{N}^n\}$ and
$$A(\alpha,q) = \sum_{s= 0}^{1-\langle\alpha,\alpha\rangle}t_{\alpha,s}\,q^s,$$
where $t_{\alpha,s}\in \mathbb{Z}$ and $\langle-,-\rangle$ is the Euler form defined by $\Gamma$.
The following identity holds in $\mathbb{Q}(q)[[X_1,\cdots,X_n]]$:
\begin{align*}
%\sum_{\pi\in\mathcal{P}^n} \!
%\frac{q^{\sum_{1\le i, j \le n}\!a_{ij}(|\pi_i, \pi_j|)}\!\prod_{s\ge 1}\!r(d_{\pi}^{\,\underline{s}} , q)}
%{\prod_{1\le i \le n}\!q^{\langle \pi_i, \pi_i\rangle} b_{\pi_i}\!(q^{-1})}X^{|\pi|} 
P(X_1,\cdots,X_n, q) 
= \!\prod_{\alpha\in\Delta^+}\!\!\prod_{s= 0}^{1-\langle\alpha,\alpha\rangle}\prod_{i=0}^\infty(1-q^{s+i}X^\alpha)^{t_{\alpha,s}}.
\end{align*}
\end{thm}
\begin{proof}
This is the counterpart of Theorem 4.9 from Hua \cite{JH 2000}, same arguments apply.
\end{proof}

Kac conjecture now a theorem confirms that Theorem 4.9 of Hua \cite{JH 2000} is a $q$-deformation of Weyl-Kac denominator identity, thus Theorem \ref{gwki} here may also be regarded as a $q$-deformation of Weyl-Kac denominator identity for some generalized Kac-Moody algebra. Defining such algebras would be a very interesting problem.

In view of Lemma 2.9 of Bozec, Schiffmann \& Vasserot \cite{B-S-V 2018}, assuming $r(\alpha,q)\in\mathbb{Q}[q]$ is equivalent to assuming $r(\alpha,q)\in\mathbb{Z}[q]$.
\begin{cjc}\label{cjc}
Under the assumption that $r(\alpha, q)$ exists and $r(\alpha,q)\in\mathbb{Z}[q]$ for all $\alpha\in\mathbb{N}^n$,  all coefficients of polynomial $A(\alpha,q)$ are non-negative integers.
\end{cjc}

\section{Special Cases}

\textbf{Case 1.} Let $\mathcal{R}$ be an empty set. Every representation of $\Gamma$ respects $\mathcal{R}$ nilpotently, thus  
$\text{Rep}(\alpha, \mathbb{F}_q)_\mathcal{R} =\text{Rep}(\alpha, \mathbb{F}_q)$. Since $r(\alpha, q) = q^{\sum_{1\le i,j \le n}a_{ij}\alpha_i\alpha_j}$ for $\alpha\in\mathbb{N}^n$, $r(d_{\pi}^{\,\underline{s}} , q) =q^{ \sum_{1\le i,j \le n} a_{ij}m_{\pi_i}^{\underline{s}}m_{\pi_j}^{\underline{s}}}$ for $\pi\in\mathcal{P}^n$  and $s\in\mathbb{N}\backslash\{0\}$. Thus,
\begin{align*}
P(X_1,\cdots,X_n, q) &= \sum_{\pi\in\mathcal{P}^n} \!
\frac{q^{\sum_{1\le i, j \le n}\!a_{ij}(|\pi_i, \pi_j|)} \prod_{s\ge 1}\!\!\left(q^{\sum_{1\le i, j \le n} a_{ij}m_{\pi_i}^{\underline{s}}m_{\pi_j}^{\underline{s}} }\right) }
{\prod_{1\le i \le n}\!q^{\langle \pi_i, \pi_i\rangle} b_{\pi_i}\!(q^{-1})}X^{|\pi|} \\
&= \sum_{\pi\in\mathcal{P}^n} \! \frac{q^{\sum_{1\le i, j \le n}\!\left(a_{ij}(|\pi_i, \pi_j|) +\sum_{s\ge 1}a_{ij}m_{\pi_i}^{\underline{s}}m_{\pi_j}^{\underline{s}}\right)}}
{\prod_{1\le i \le n}\!q^{\langle \pi_i, \pi_i\rangle} b_{\pi_i}\!(q^{-1})}X^{|\pi|} \\
&=\sum_{\pi\in\mathcal{P}^n} \!
\frac{q^{\sum_{1\le i, j \le n}\!a_{ij}\langle\pi_i, \pi_j\rangle}}
{\prod_{1\le i \le n}\!q^{\langle \pi_i, \pi_i\rangle} b_{\pi_i}\!(q^{-1})}X^{|\pi|}.
\end{align*}
Thus Theorem \ref{burnside} is equivalent to Theorem 4.9 of Hua \cite{JH 2000}.

\textbf{Case 2.} Let $\Gamma$ be the quiver with one vertex and $g$ edge-loops, i.e., the quiver defined by the matrix $[g]$ where $g\ge 1$, and
$\mathcal{R} = \{X_{11}^{(i)} : 1 \le i \le g\}$. The isomorphism classes of representations of $\Gamma$ over $\mathbb{F}_q$ that respect $\mathcal{R}$ nilpotently are in one-to-one
correspondence with the orbits of $g$-tuples of nilpotent matrices over $\mathbb{F}_q$ under simultaneous conjugation. Since the number of $n\times n$ nilpotent matrices over $\mathbb{F}_q$ is
$q^{n^2-n}$ according to Fine \& Herstein \cite{F-H 1958}, $r(n, q) = q^{g(n^2-n)}$ for $n\in\mathbb{N}$. Thus,
\begin{align*}
P(X, q) &= \sum_{\pi\in\mathcal{P}} \!
\frac{q^{g(|\pi, \pi|)} \prod_{s\ge 1}\!q^{g\left((m_{\pi}^{\underline{s}})^2 - m_{\pi}^{\underline{s}}\right) }}
{q^{\langle \pi, \pi\rangle} b_{\pi}(q^{-1})}X^{|\pi|} \\
&= \sum_{\pi\in\mathcal{P}} \!
\frac{q^{g(|\pi, \pi|) + \sum_{s\ge 1} g(m_{\pi}^{\underline{s}})^2 - \sum_{s\ge 1} g\left(m_{\pi}^{\underline{s}}\right) }}
{q^{\langle \pi, \pi\rangle} b_{\pi}(q^{-1})}X^{|\pi|} \\
&= \sum_{\pi\in\mathcal{P}} \!
\frac{q^{g(\langle\pi, \pi\rangle -l(\pi))}}
{q^{\langle \pi, \pi\rangle} b_{\pi}(q^{-1})}X^{|\pi|},
\end{align*}
where $\l(\pi) = \sum_{s\ge 1}m_{\pi}^{\underline{s}}$ is the length of $\pi$. Thus Theorem \ref{burnside} is equivalent to Theorem 4.1 of Hua \cite{JH 2021}.

\textbf{Case 3.} Let $\Gamma$ be the quiver defined by the following matrix, where $g$ is an integer greater than 0:
$$
\left[
\begin{array}{cc}
0 & 1  \\ 
g & 0 
\end{array}
\right],
$$
and $\mathcal{R} = \{ X_{12}^{(1)}X_{21}^{(i)} : 1 \le i \le g \}$ a set of cyclic relations for $\Gamma$.

Let $[n]_q = \prod_{i=0}^{n-1}(q^n - q^i)$ for $n\ge 1$ and $[0]_q=1$. Thus $|\text{GL}(n,\mathbb{F}_q)| = [n]_q$.
Given a dimension vector  $(m,n)\in\mathbb{N}^2$ and a non-negative integer $r\le \min(m,n)$, let $D_{(m,n,r)}$ be the $m\times n$ matrix in the following form:
$$ \left[
\begin{array}{cc}
I & 0  \\ 
0 & 0 
\end{array}
\right], $$
where $I$ is the identity matrix of order $r$.

Let $\mathcal{C}_{(m,n,r)}$ be the centralizer of $D_{(m,n,r)}$ in $\text{GL}((m,n),\mathbb{F}_q)$, i.e.,
$$
\mathcal{C}_{(m,n,r)} = \left\{(M, N)\in\text{GL}((m,n),\mathbb{F}_q) : M^{-1} D_{(m,n,r)} N = D_{(m,n,r)} \right \},
$$
and hence the number of $m\times n$ matrices over $\mathbb{F}_q$ which have rank $r$ is equal to ${|\text{GL}((m,n),\mathbb{F}_q)|}/{|\mathcal{C}_{(m,n,r)}|}$.
For any $(M, N)\in\text{GL}((m,n),\mathbb{F}_q)$, $M$ and $N$ can be written as block matrices as follows:
$$M = \left[
\begin{array}{ll}
A_{r\times r} & B_{r\times (m-r)}  \\ 
C_{(m-r) \times r} & D_{(m-r)\times (m-r)}
\end{array}
\right], 
N = \left[
\begin{array}{ll}
E_{r\times r} & F_{r\times (n-r)}  \\ 
G_{(n-r) \times r} & H_{(n-r)\times (n-r)}
\end{array}
\right],$$
where the orders of the submatrices are indicated by their subscripts. Since
$$M^{-1} D_{(m,n,r)} N = D_{(m,n,r)} \text{ if and only if $A=E$, $C=0$ and $F=0$}, $$
it follows that
\begin{align*}
\left|\mathcal{C}_{(m,n,r)} \right| 
= \, & [r]_q [m-r]_q q^{r(m-r)} [n-r]_q q^{(n-r)r} \\
= \, & [r]_q [m-r]_q [n-r]_q q^{r(m+n) - 2r^2}.
\end{align*}

Let
$$
\mathcal{N}_{(m,n,r)} = \left\{N\in\textup{Mat}(n\times m, \mathbb{F}_q) : D_{(m,n,r)} N \textit{ is nilpotent } \right \}.
$$
Any $n\times m$ matrix $N$ can be written as a block matrix as follows:
$$ \left[
\begin{array}{ll}
A_{r\times r} & B_{r\times (m-r)}  \\ 
C_{(n-r) \times r} & D_{(n-r)\times (m-r)}
\end{array}
\right],
$$
where the orders of the submatrices are indicated by their subscripts. $D_{(m,n,r)} N$ is nilpotent if and only if $A$ is nilpotent, therefore
$$
|\mathcal{N}_{(m,n,r)} | = q^{r^2-r}q^{mn-r^2} = q^{mn-r}.
$$

Let 
$$\mathcal{E}_{(m,n,r)} = \{\sigma\in\text{Rep}((m,n), \mathbb{F}_q)_{\mathcal{R}}:\sigma(X_{12}^{(1)}) \textit{ has rank } r \}.$$
Since the number of $m\times n$ matrices over $\mathbb{F}_q$ that have rank $r$ is equal to ${|\text{GL}((m,n),\mathbb{F}_q)| }/{|\mathcal{C}_{(m,n,r)}|}$,
\begin{align*}
|\mathcal{E}_{(m,n,r)}| &= \frac{|\text{GL}((m,n),\mathbb{F}_q)| }{|\mathcal{C}_{(m,n,r)}|} |\mathcal{N}_{(m,n,r)} |^g
 = \frac{[m]_q [n]_q q^{g(mn-r)}}{[r]_q [m-r]_q [n-r]_q q^{r(m+n) -2r^2}}.
\end{align*}

Since $\text{Rep}((m,n), \mathbb{F}_q)_{\mathcal{R}}$ is a disjoint union of $\mathcal{E}_{(m,n,r)}$ where $0\le r \le \min(m,n)$,
\begin{equation}\label{gen kronecker}
r((m,n), q) = \sum_{r=0}^{\min(m,n)} \frac{[m]_q [n]_q q^{g(mn-r)}}{[r]_q [m-r]_q [n-r]_q q^{r(m+n) -2r^2}}.
\end{equation}

It follows that $r((m,n), q)$ is a polynomial  in $q$ with integral coefficients and hence $A((m,n),q)$ can be calculated by Theorem \ref{A poly}.

When $g=1$, $\Gamma$ is the quiver below known as affine Dynkin quiver $\tilde{A}_1$:
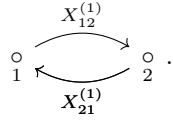
\[
\begin{tikzcd}[row sep=large, column sep = large]
\underset{1}{\circ}\arrow[r, bend left, "{X_{12}^{(1)}}"] & 
\arrow[l, bend left, "{X_{21}^{(1)}}" ]
\arrow[l, bend left, "{X_{21}^{(1)}}" ] \underset{2}{\circ} 
\end{tikzcd}.
\]
Thus we have
$$
P(X_1,X_2, q) = \sum_{\pi\in\mathcal{P}^2} 
\frac{q^{2(|\pi_1, \pi_2|)}\prod_{s\ge 1}\!r(d_{\pi}^{\,\underline{s}} , q)}
{\prod_{1\le i \le 2}q^{\langle \pi_i, \pi_i\rangle} b_{\pi_i}\!(q^{-1})}\,X^{|\pi|},
$$
where $r(d_{\pi}^{\,\underline{s}} , q)$ is given by identity (\ref{gen kronecker}) with $g=1$. 

For any $(m,n)\in\mathbb{N}^2\backslash\{(0,0)\}$,
according to Donovan \& Freislich \cite{D-F 1973} and Dlab \& Ringel \cite{D-R 1976}, $A((m,n),q)$ has the following form:
\begin{align*}
    A((m,n),q) =
    \begin{cases}
      2  & \textit{ if } |m-n| = 0, \\
      1  & \textit{ if } |m-n| = 1, \\
	0  & \textit{ if } |m-n| > 1.
    \end{cases}
\end{align*}
Thus Theorem \ref{gwki} amounts to the following identity:
$$
P(X_1,X_2, q) = \prod_{n=1}^\infty\prod_{i=0}^\infty(1-q^iX_1^nX_2^{n-1})(1-q^iX_1^{n-1}X_2^n)(1-q^iX_1^nX_2^n)^2.
$$

In all cases above, $r(\alpha, q)$'s are known polynomials in $q$ with integral coefficients, thus $A(\alpha, q)$'s are computable by Theorem \ref{A poly}. All sample results given in Hua \cite{JH 2000}\cite{JH 2021}
are consistent with the conjecture above.

\vspace{0.2cm}
\textbf{\large{Acknowledgments}}
\vspace{0.1cm}

The authors would like to thank Xueqing Chen for his constructive comments and suggestions on the draft of this paper.

\vspace{0.2cm}
%\newpage
Department of Mathematical Sciences, Tsinghua University, Beijing 100084, China.\\
\textit{Email address}: \texttt{bmdeng@math.tsinghua.edu.cn} \\ \\
Mathematics Enthusiast \\
\textit{Email address}: \texttt{jiuzhao.hua@gmail.com}
\end{document}